\documentclass[a4paper,reqno,12pt]{amsart}

\usepackage{fullpage}

\usepackage{mathtools}

\usepackage{dsfont}

\usepackage[plainpages=false,colorlinks,linkcolor=black,citecolor=black,urlcolor=black,breaklinks]{hyperref}

\usepackage{amsmath,amsthm,amssymb}
\usepackage{mathrsfs}
\usepackage[shortlabels]{enumitem}
\usepackage{graphicx}
\usepackage[font=small]{caption}
\usepackage{xcolor}
\usepackage{url}

\newcommand{\N}{\mathbb{N}}
\newcommand{\R}{{\mathbb{R}}}

\newcommand{\C}{{\mathbb{C}}}
\newcommand{\Z}{{\mathbb{Z}}}

\newcommand{\dd}{{{\rm d}}}
\newcommand{\ii}{{\rm i}}

\newcommand{\ov}{\overline}

\newcommand\wh{\widehat}

\newcommand{\eps}{\varepsilon}

\newcommand{\Dom}{{\operatorname{Dom}}}

\newcommand{\Ran}{{\operatorname{Ran}}}

\newcommand{\Rank}{{\operatorname{rank}}}
\renewcommand{\Re}{\operatorname{Re}}
\renewcommand{\Im}{\operatorname{Im}}

\newcommand{\supp}{\operatorname{supp}}

\newcommand{\BigO}{\mathcal{O}}

\newcommand{\Rd}{\mathbb{R}^d}

\theoremstyle{plain}

\newtheorem{theorem}{Theorem}[section]
\newtheorem{lemma}[theorem]{Lemma}

\newtheorem{proposition}[theorem]{Proposition}

\theoremstyle{definition}
\newtheorem{example}[theorem]{Example}
\newtheorem{remark}[theorem]{Remark}
\newtheorem{asm-sec}[theorem]{Assumption}

\newcommand\cB{\mathcal B}

\newcommand\cD{\mathcal D}

\newcommand\cH{\mathcal H}

\newcommand\cR{\mathcal R}
\newcommand\cS{\mathcal S}

\numberwithin{equation}{section}
\numberwithin{figure}{section}

\usepackage[
backend=biber
,style=alphabetic
,firstinits=true
,maxnames=4 
,sorting=nyvt
,isbn=false
]{biblatex}

\renewbibmacro{in:}{} 

\DeclareFieldFormat[article,inbook,incollection,inproceedings,patent,thesis,unpublished,misc]{title}{#1}
\DeclareFieldFormat{pages}{#1}

\AtEveryBibitem{
	\ifentrytype{article}{
		\clearfield{number}
	}{}
}

\begin{document}

\author{Boris Mityagin}

\address[Boris Mityagin]{Department of Mathematics, The Ohio State University, 231 West 18th Ave,
	Columbus, OH 43210, USA}
\email{mityagin.1@osu.edu, boris.mityagin@gmail.com}

\author{Petr Siegl}

\address[Petr Siegl]{Institute of Applied Mathematics, Graz University of Technology, Steyrergasse 30, 8010 Graz, Austria}
\email{siegl@tugraz.at}




\date{\today}

\title{Riesz property in the case of multiple eigenvalues}

\begin{abstract}
We analyze spectra and the Riesz property of spectral projections of non-symmetric perturbations of self-adjoint operators with eigenvalues having arbitrary multiplicities, including infinite ones. In particular, we establish the Riesz property for perturbations of the multi-dimensional harmonic oscillator, Landau Hamiltonian and Laplace-Beltrami operator on a sphere by complex-valued $L^r$-potentials if $d/2 < r < \infty$.
\end{abstract}

\maketitle

\section{Introduction}

Let $A$ be a self-adjoint operator in a separable Hilbert space $\cH$ and assume that the spectrum of $A$ satisfies 
\begin{equation}\label{A.sp.intro}
\sigma(A) = \{ \mu_k \equiv \mu_k(A)\}_{k \in \N}, \quad 0< \mu_1 < \mu_2 < \dots, \quad \mu_k \nearrow + \infty.	
\end{equation}
We denote by $P_k^0 \equiv P_k^0(A) $ the spectral projections on the invariant subspaces associated with $\mu_k$, namely,
\begin{equation}
P_k^0 = \frac{1}{2\pi \ii} \int_{|z-\mu_k|=r_k}(z-A)^{-1} \dd z, \quad k \in \N,
\end{equation}
where
\begin{equation}\label{rk.def}
	r_1  := \frac{1}{2} (\mu_{2}-\mu_1), \quad  r_k := \frac 12 \min\{\mu_k-\mu_{k-1}, \mu_{k+1}-\mu_k\}, \quad k \geq 2. 
\end{equation}

We analyze spectra and most importantly the Riesz projections of the perturbed operator (defined as a form sum)
\begin{equation}
	T = A + V
\end{equation}
where the perturbations $V$ are non-symmetric and unbounded in general. Our goal is to identify conditions on $V$ such that a  system of the Riesz projections of $T$ is complete and has a Riesz property, i.e.~it is equivalent to a complete system of orthonormal projections via a similarity transform.

Our main motivation is perturbations by $L^r$-complex potentials $V$ of multi-dimensional differential operators. In particular, for $d>1$, we consider 
\begin{enumerate}[\upshape (i), wide]
\item 
the $d$-dimensional harmonic oscillator in $L^2(\Rd)$
\begin{equation}\label{HO.intro}
	A_{\rm HO} = -\Delta + |x|^2,
\end{equation}
having the spectrum 
\begin{equation}\label{HO.sp.intro}
\sigma(A_{\rm HO}) = \{2 k -2 + d \}_{k \in \N}, \quad \Rank(P_k^0) = \binom{k + d-3}{k-1}, \quad k \in \N;	
\end{equation}
\item the Landau Hamiltonian (twisted Laplacian) in $L^2(\Rd)$ with $d \in 2 \N$,
\begin{equation}\label{Lan.intro}
	A_{\rm Lan} = \sum_{j=1}^\frac d2 
	\left(
	-\ii \partial_{x_j} + \frac{y_j}{2}
	\right)^2 + 
	\left(
	-\ii \partial_{y_j} + \frac{x_j}{2}
	\right)^2,
\end{equation}
where 
\begin{equation}\label{Lan.sp.intro}
\sigma(A_{\rm Lan}) = \{2 k -2 + d/2 \}_{k \in \N}, \quad \Rank(P_k^0) = \infty, \quad k \in \N;
\end{equation}
\item
 the Laplace-Beltrami operator on $\mathbb S^d$ in $L^2(\mathbb S^d)$ 
\begin{equation}\label{Sd.intro}
A_{\mathbb S^d} =  - \Delta_{\mathbb S^d}
\end{equation}
where
\begin{equation}
\sigma(A_{\mathbb S^d}) = \{ (k-1)(k-2+d) \}_{k \in \N}, 
\quad \Rank(P_k^0) = 
\frac{(2k-3+d)(k-3+d)!}{(k-1)!(d-1)!}, \quad k \in \N.
\end{equation}
\end{enumerate}
For more details and references see Section~\ref{sec:diff.op}. Notice that (for $d>1$) the operators \eqref{HO.intro}, \eqref{Lan.intro} and \eqref{Sd.intro} have eigenvalues with unbounded multiplicities as $k \to \infty$, in the case of $A_{\rm Lan}$ even infinite ones.

In all three cases \eqref{HO.intro}, \eqref{Lan.intro} and \eqref{Sd.intro}, we consider perturbations by complex-valued potentials $V \in L^r(\Rd)$ for $r$ satisfying 
\begin{equation}\label{r.intro}
	\frac d2 < r < \infty. 
\end{equation}
We first show that the perturbed operator $T$ has spectrum localized in a union of a  box around $0$ and non-intersecting disks around sufficiently large unperturbed eigenvalues $\mu_k$. More precisely, 
\begin{equation}\label{Sp.loc.intro}
\sigma(T) \subset \Pi_0\cup \bigcup_{k > N_0} \Pi_k,
\end{equation}
where 
\begin{equation}\label{Pi.def}
	\begin{aligned}
		\Pi_0 &\equiv \Pi_0(N,h_1,h_2):= \left\{z \in \C \, : \,  -h_1 <\Re z \leq \mu_N + r_N, \ |\Im z| \leq h_2 \right\}, 
		\\ \Pi_k&:= B_{r_k}(\mu_k), \quad k \in \N, \quad \Pi:=\bigcup_{k \in \N} \Pi_k.
	\end{aligned}
\end{equation}
with sufficiently large $h_1, h_2, N_0>0$. 

In fact, a more precise eventual localization in the discs with radii depending on $d$ and $r$ is possible, see Section~\ref{sec:diff.op} and Remark~\ref{rem:EV.loc} for details; in particular, this matches the localization result \cite[Thm.~2.1]{Cuenin-2017-88} for $A_{\rm Lan}+V$ if \eqref{r.intro} holds. We note that \cite[Thm.~2.1]{Cuenin-2017-88} contains a localization of eigenvalues also in the critical case $r=d/2$ for $d>2$ and establishes its sharpness, cf.~also previous results on spectra in \cite{Cuenin-2017-42}.

Based on the spectral localization \eqref{Sp.loc.intro}, we define the disjoint Riesz projections of $T$
\begin{equation}\label{Pn.SN.intro}
	\begin{aligned}
		S_0:=\frac{1}{2\pi \ii} \int_{\partial \Pi_0}(z-T)^{-1} \dd z, \quad 
		P_k  := \frac{1}{2\pi \ii} \int_{ \partial \Pi_k }(z-T)^{-1} \dd z, \quad k> N_0,
	\end{aligned}
\end{equation}
and show that
\begin{equation}\label{rank.SN.intro}
\Rank  (S_0)  = \Rank \left(\sum_{k=1}^{N_0} P_k^0 \right), \quad	\Rank  (P_k)  = \Rank (P_k^0), \quad  k>N_0.
\end{equation}
As a main result, we obtain that the system $\{S_0\} \cup \{P_k\}_{k>N_0}$ has a Riesz property, namely, there exists a boundedly invertible $W \in \cB(\cH)$ such that  
\begin{equation}\label{S0Pk.Riesz.intro}
	S_0 = W^{-1} \left(\sum_{k=1}^{N_0} P_k^0 \right) W, \qquad P_k = W^{-1} P_k^0 W, \quad k > N_0,
\end{equation}
cf.~Theorems~\ref{thm:HO}, \ref{thm:Landau} and \ref{thm:LB}. We remark that one obtains the Riesz property also for bounded perturbations $V \in L^\infty(\Rd)$ of $A_{\mathbb S^d}$ and, by \cite[Thm.~1]{Motovilov-2017-8}, for bounded perturbations $V \in L^\infty(\Rd)$ with $\|V\|_{L^\infty}<1$ of $A_{\rm HO}$ and $A_{\rm Lan}$. On the other hand, it remains open whether the Riesz property is absent for perturbations $V \in L^r(\Rd)$ with $r< d/2$; $r = d/2$ and a sufficiently large $\|V\|_{L^r}$; or $r = \infty$ and $\|V\|_{L^\infty} \geq 1$.

As an additional result for $A_{\mathbb S^d}$, in Theorem~\ref{thm:LB.delta} we analyze perturbations of $A_{\mathbb S^d}$ by $\delta$-potentials supported on a $(d-1)$-dimensional submanifold and arrive at similar conclusions as for perturbations by $V \in L^r(\Rd)$.

For comparison, we recall below the known results on one-dimensional harmonic oscillator $-\partial_x^2+x^2$ in $L^2(\R)$ and the Laplace-Beltrami operator on the circle, which corresponds to the Hill operator 
\begin{equation}
	A_{\rm Hill} = - \partial_x^2, \quad \Dom(A_{\rm Hill}) = \{ f \in H^2(-\pi,\pi) \, : \, f(\pi) = f(-\pi), f'(\pi) = f'(-\pi)\}.
\end{equation}
(For a broader collection of results on the Riesz property for one-dimensional differential operators (Dirac, Hill, Schr\"odinger, damped wave) see~\cite{Cox-1994-19,Djakov-2011-351,Djakov-2010-283,Djakov-2012-263,Gesztesy-2012-253,Lunyov-2016-441,Shkalikov-2016-71} and references therein.)
Namely, for $T = A_{\rm HO} + V$ for $d=1$ it is known that the system $\{S_0\} \cup \{P_k\}_{k>N_0}$ has a Riesz property if $V = V_1 + V_2 + V_3 + V_4$ where 
\begin{equation}
\begin{aligned}
& V_1 \in L^r(\R), \quad r \in [1,\infty),
\\
& V_2 \in H^{-s}(\R), \quad s \in (0,1/4),
\\
& V_3 \in  H^{-t}(\R), \quad t \in (0,1/2) \text{ and $\supp(V_3)$ is compact},  
\\
& V_4 = \sum_{k \in \Z} \nu_k \delta(x-x_k), \ \text{with } \{\nu_k\}_{k \in \Z} \in \ell^1(\Z) \ \text{and} \ \{x_k\}_{k \in \Z} \subset \R;
\end{aligned}
\end{equation}
details can be found in \cite{Adduci-2012-10,Adduci-2012-73,Mityagin-2016-106,Mityagin-2019-139}. For $T = A_{\rm Hill} + V$ the Riesz property of the system $\{S_0\} \cup \{P_k\}_{k>N_0}$ holds if $V \in H^{-1}(-\pi,\pi)$, see \cite{Djakov-2009-481}; notice that $\Rank(P_k) =2$ for $k >N_0$.

Our method relies on an abstract test of the Riesz property for locally-form subordinated perturbations of a self-adjoint operator $A$ for which \eqref{A.sp.intro} holds. Namely, if the form  $v : \Dom(A^\frac 12) \times \Dom(A^\frac 12) \to \C$ satisfies
\begin{equation}\label{asm:v.loc.sub}
	|v(P_j^0 f, P_k^0 g)| \leq \omega_j \omega_k \|P_j^0 f\| \|P_k^0 g\|, \quad j,k \in \N, \ f,g \in \Dom(A^\frac12),
\end{equation}
and for the sequence $\{\omega_j\}_{j \in \N}$ the condition
\begin{equation}\label{asm:omega}
	\sum_{j=1}^\infty \frac{\omega_j^2}{ \mu_j} < \infty, \quad \sum_{\substack{j=1 \\j \neq n}}^\infty \frac{\omega_j^2}{|\mu_n-\mu_j|} + \frac{\omega_n^2}{r_n} = o(1), \quad n \to \infty,
\end{equation}
holds, then the spectrum of $T= A+V$, defined as a form sum, is localized as in \eqref{Sp.loc.intro}, see Proposition~\ref{prop:loc}, and the system $\{S_0\} \cup \{P_k\}_{k>N_0}$ has a Riesz property, see Theorem~\ref{thm:RB}. 

Perturbation results based on local-subordination were developed in \cite{Adduci-2012-10,Shkalikov-2010-269,Adduci-2012-73,Mityagin-2016-106,Shkalikov-2016-71,Mityagin-2019-139,Mityagin-2025-31}. Comparing to the most recent  \cite[Thm.~3.4]{Mityagin-2025-31}, the novelty of Theorem~\ref{thm:RB} lies in arbitrary multiplicities of the eigenvalues $\mu_k$ and, more importantly, in removing the condition $A^{-1}$ being in a Schatten class $\cS_p$ for some $p>0$. The latter was used in \cite{Mityagin-2025-31} to establish the completeness of the Riesz projections of $T$. Here we give a new proof of the completeness, inspired by the strategy in \cite[Thm.~1]{Motovilov-2017-8}. These extensions make Theorem~\ref{thm:RB} applicable in particular for perturbations of Landau Hamiltonian, where the multiplicities of $\mu_k$ are infinite. 

The condition \eqref{asm:omega} can be checked easily if the eigenvalue gaps and sequence $\{\omega_k\}_{k \in \N}$, have power-estimates, which often appears for differential operators. Examples of (abstract) operators and perturbations with $\{\omega_k\}_{k \in \N}$ without power-estimates and still satisfying \eqref{asm:omega} can be found in \cite{Mityagin-2025-31}.

\begin{example}\label{ex:mu.om.power}
	Let for some constants $C>0$, $\gamma>0$ and $\alpha \in \R$
	\begin{equation}\label{omega.reg.1}
		\mu_{k+1} - \mu_k \geq C  k^{\gamma-1} \quad \text{and} \quad \omega_k = \BigO(k^{-\alpha}), \quad k \to \infty.	
	\end{equation}
	Then the conditions \eqref{asm:omega} are satisfied if
	\begin{equation}\label{omega.reg.2}
		2\alpha+\gamma>1.	
	\end{equation}
	Indeed, this follows by
	\begin{equation}\label{om.ex.power}
		\sum_{\substack{j=1 \\j \neq n}}^\infty \frac{j^{-2 \alpha}}{|\mu_n-\mu_j|} 
		= 
		\begin{cases}
			\BigO(n^{-(2\alpha+\gamma-1)} \log n),  &\alpha \leq 1/2,
			\\[1mm]
			\BigO(n^{-\gamma}), & \alpha >1/2,
		\end{cases}
		\qquad n \to \infty,
	\end{equation}  
	see \cite[Lemma~4.1, 4.3]{Mityagin-2019-139}, and
	\begin{equation}\label{om.ex.power.diag}
		\frac{\omega_n^2}{r_n}  = \BigO(n^{2\alpha + \gamma -1}), \qquad n \to +\infty.
	\end{equation}  
\end{example}

For perturbations by potentials $V \in L^r$, $r \geq 1$ , the H\"older inequality yields
\begin{equation}\label{V.Hold}
	| \langle V f_j, g_k \rangle| \leq \|V\|_{L^r} \|f_j\|_{L^{p}}\|g_k\|_{L^{p}},
\end{equation}
where
\begin{equation}\label{V.hold.rp}
\frac 1r + \frac 2p =1.
\end{equation}
Hence the sequence $\{\omega_j\}_{j \in \N}$ in \eqref{asm:v.loc.sub} can be chosen as 
\begin{equation}\label{Pj.ue}
	\|P_j^0\|_{L^2 \to L^p}.
\end{equation}
The upper estimates on the norms \eqref{Pj.ue} for \eqref{HO.intro}, \eqref{Lan.intro} and \eqref{Sd.intro} were studied in \cite{Koch-2005-128,Koch-2007-8,Wang-2025-378,Stempak-1998-76,Koch-2007-180,Sogge-1988-77} and they can be used to find sufficient conditions on $r \in [1,\infty]$ so that \eqref{omega.reg.2}, and hence \eqref{asm:omega}, holds. A similar strategy is used also for perturbations of \eqref{Sd.intro} by $\delta_\Sigma$-potentials supported on submanifolds of codimension~$1$, cf.~Subsection~\ref{ssec:LB} and \cite{Burq-2007-138}.

\section{Riesz system test}

\subsection{The perturbed operator $T$}

The operator $T$ is defined below as a relatively form-bounded perturbation of a self-adjoint operator $A$ for which the spectrum satisfy \eqref{A.sp.intro}.
The operator $A$ is associated with the non-negative closed form
\begin{equation}\label{a.def}
	a[f]:=a(f,f) = \|A^\frac12 f\|^2, \quad f \in \Dom(a):= \Dom(A^\frac12), 
\end{equation}
see the representation theorems in \cite[Thm.~VI.2.1, VI.2.23]{Kato-1966}.

The next lemma shows that $v$ is a relatively form bounded perturbation of $a$ with bound $0$. The operator $T$ is defined as the unique m-sectorial operator associated with the closed sectorial form 
\begin{equation}\label{t.def}
t := a + v, \qquad \Dom(t) := \Dom(a).
\end{equation}

\begin{lemma}\label{lem:t.def}
Let $A$ be a self-adjoint operator in a separable Hilbert space, let the spectrum of $A$ satisfy \eqref{A.sp.intro} and let $a$ be the associated form in \eqref{a.def}. Suppose that a sesquilinear form $v$ satisfies \eqref{asm:v.loc.sub} and 
\begin{equation}\label{asm:om.l1}
	\sum_{j=1}^\infty \frac{\omega_j^2}{ \mu_j} < \infty.
\end{equation}
Then for any $\eps>0$, there exists $C_\eps>0$ such that for all $f \in \Dom(a)$
\begin{equation}\label{v.eps.a}
	|v[f]| \leq \eps (a[f] + C_\eps\|f\|^2).
\end{equation}
Thus the form $t$ in \eqref{t.def} is closed and sectorial and hence it defines an m-sectorial operator $T$ (via the first representation theorem~\cite[Thm.~VI.2.1]{Kato-1966}). 
\end{lemma}
\begin{proof}[Sketch of the proof]
The proof follows closely the one of \cite[Prop.~3.1]{Mityagin-2025-31}. The needed adjustments amount to expansions using the complete system of orthonormal projections $P_k^0$, $k \in \N$, namely,
\begin{equation}\label{f.Da.exp}
	f = \sum_{k=1}^\infty P_k^0 f, \quad \|f\|^2 = \sum_{k=1}^\infty \|P_k^0 f\|^2, \quad a[f] = \|A^\frac12 f\|^2 =  \sum_{k=1}^\infty \mu_k \| P_k^0 f|^2.
\end{equation}
\end{proof}

To factorize the resolvent of $T$ in a suitable way, one introduces the operators 
\begin{equation}\label{K.def}
	K(z) f := \sum_{k=1}^\infty (z-\mu_k)^{-\frac 12} P_k^0 f, \quad z \in \rho(A), \quad f \in \cH,
\end{equation}
where, for $0\neq w \in \C$ and $s\in \R$, the $s$-power of $w$ is taken as $w^s := |w|^s e^{\ii s \arg w}$ with $-\pi < \arg w \leq \pi$. Notice that
\begin{equation}
	\|K(z)\| = \max_{k \in \N}{ \frac{1}{|z-\mu_k|^\frac12}}, \quad z \in \rho(A), 
\end{equation}
$\Ran(K(z)) = \Dom(A^{1/2})$ and by \eqref{K.def}
\begin{equation}\label{K^2}
	K(z)^2 = (z-A)^{-1}, \quad z \in \rho(A).
\end{equation}
Moreover, let $B(z)$, $z \in \rho(A)$, be a bounded operator uniquely determined by the bounded form 
$v(K(z) \cdot,K(z)^* \cdot)$; i.e.~$B(z)$ satisfies
\begin{equation}\label{B.form.def}
	\langle B(z)f,g \rangle = v(K(z) f,K(z)^* g), \qquad f,g \in \cH. 
\end{equation}
Since for all $f \in \cH$, 
\begin{equation}\label{B.norm.detail}
\begin{aligned}
\|B(z) f\| & = \sup_{0\neq g \in \cH} \frac{|\langle B(z) f,g \rangle |}{\|g\|} 
\leq \sup_{0\neq g \in \cH} \frac{1}{\|g\|} \sum_{j,k \in \N}  \frac{|v(P_j^0 f,P_k^0 g)|}{|z-\mu_j|^\frac12|z-\mu_k|^\frac12} 
\\
& 
\leq \sup_{0\neq g \in \cH} \frac{1}{\|g\|} \sum_{j,k \in \N}  \frac{\omega_j  \|P_j^0 f\|}{|z-\mu_j|^\frac12} \frac{\omega_k \|P_k^0 g\|}{|z-\mu_k|^\frac12} 
\\
&\leq 
\left(
\sum_{j=1}^\infty \frac{\omega_j^2}{|z-\mu_j|}
\right)^\frac12 
\|f\| 
\left(
\sum_{k=1}^\infty \frac{\omega_k^2}{|z-\mu_k|}
\right)^\frac12,
\end{aligned}
\end{equation}
we have (if \eqref{asm:om.l1} holds)
\begin{equation}\label{B.norm}
\|B(z)\| \leq \sum_{j=1}^\infty \frac{\omega_j^2}{|z-\mu_j| } < \infty.
\end{equation}

Finally, for all $z \in \rho(A)$ such that $I+B(z)$ is boundedly invertible, one obtains 
\begin{equation}\label{Tz.res.dec}
	(z-T)^{-1} = K(z)(I-B(z))^{-1}K(z);
\end{equation}
see~\cite[Chap.VI.3.1]{Kato-1966}, \cite[Lemma~1]{Agranovich-1995-28} for details.

\subsection{Localization of spectrum of $T$}

The spectrum of $T$ is eventually localized in non-intersecting disks around the eigenvalues $\mu_k$, see \eqref{Pi.def}.

\begin{proposition}\label{prop:loc}
Let $A$ be a self-adjoint operator in a separable Hilbert space and let the spectrum of $A$ satisfy \eqref{A.sp.intro}. 
Suppose that a sesquilinear form $v$ satisfies \eqref{asm:v.loc.sub} and \eqref{asm:omega} and let $T$ be the m-sectorial operator defined via the form sum of $A$ and $v$ in Lemma~\ref{lem:t.def}. Then there exist $N_0 \in \N$ and $h_1, h_2>0$ such that 
	\begin{equation}\label{T.sp.loc}
		\sigma(T) \subset \Pi_0(N_0,h_1,h_2) \cup \bigcup_{k > N_0} \Pi_k.
	\end{equation}
	With these $N_0$ and $h_1$, $h_2$, the Riesz projections 
	\begin{equation}\label{Pn.SN.def}
		\begin{aligned}
			S_0:=\frac{1}{2\pi \ii} \int_{\partial \Pi_0}(z-T)^{-1} \dd z, \quad 
			P_k  := \frac{1}{2\pi \ii} \int_{\partial \Pi_k}(z-T)^{-1} \dd z, \quad k > N_0,
		\end{aligned}
	\end{equation}
	are well-defined, disjoint and 
	\begin{equation}\label{rank.SN.Pn}
		\Rank  (S_0)  = \Rank \left( \sum_{k=1}^{N_0} P_k^0 \right), 
			\quad 
			\Rank  (P_k)  = \Rank (P_k^0), \quad  k>N_0.
	\end{equation}
\end{proposition}
\begin{proof}[Sketch of the proof]
Due to the representation of the resolvent of $T$ in \eqref{Tz.res.dec} and the estimate \eqref{B.norm}, the proof of Proposition~\ref{prop:loc} reduces to the one of \cite[Prop.~3.3]{Mityagin-2025-31} where the case $\Rank(P_n^0)=1$, $n>N_0$, is analyzed. It relies on the following bounds on $\|B(z)\|$
\begin{align}
	\sup_{\Re z \leq -h} \|B(z)\| &= o(1), \quad h \to +\infty, \label{B.est.1}
	\\
	\sup_{\substack{z \notin \Pi \\[1mm] \Re z \geq \mu_N-r_N}}\|B(z)\|	& = \BigO(\sigma_N), \quad N \to + \infty, \label{B.est.2n}
\end{align}
where
\begin{equation}\label{sigma.N.def}
	\sigma_N:= \sup_{n \geq N} \left(\sum_{\substack{j=1 \\j \neq n}}^\infty \frac{\omega_j^2}{|\mu_n-\mu_j|} + \frac{\omega_n^2}{r_n} \right), \quad N \in \N,
\end{equation}
($\sigma_N = o(1)$ as $N \to +\infty$ by \eqref{asm:omega}) and, for any $t>0$ fixed,
\begin{equation}\label{B.est.3}
	\sup_{\substack{z \notin \Pi \\[1mm] |\Re z| \leq t, |\Im z| \geq Y} } \|B(z)\|  = o(1), \quad Y \to +\infty;
\end{equation}
for details see \cite[Prop.~3.1, Lemma~4.1]{Mityagin-2025-31} and \eqref{B.norm}.
The stability of $\Rank$ in a possibly infinite case can be justified by \cite{Krein-1948-11}, \cite[Lem.~I.3.1]{Gohberg-1969}.  
\end{proof}

\begin{remark}\label{rem:EV.loc}
The localization of the spectrum in Proposition~\ref{prop:loc} is used later in the analysis of the spectral projections, where it is convenient to work with $\Pi_k$, $k>N_0$. Nonetheless, a slight adjustment in the estimates in \cite[Lemma~4.1]{Mityagin-2025-31} shows that the spectrum of $T$ can be localized more precisely. Namely, for any $\eps >0$, there exist $N_0 \in \N$ and $h_1, h_2>0$ such that
\begin{equation}\label{T.sp.loc.1}
\sigma(T) \subset \Pi_0(N_0,h_1,h_2) \cup \bigcup_{k > N_0} B_{(1+\eps) \omega_k^2}(\mu_k).
\end{equation}
In particular, for the sequences in Example~\ref{ex:mu.om.power} and a sufficiently large $\varrho>0$, the spectrum of $T$ is eventually localized in the disks 
$B_{\varrho k^{-2 \alpha}}(\mu_k)$.
\end{remark}

\subsection{Riesz system test}

To analyze the differential operators in Section~\ref{sec:diff.op}, we extend the abstract Riesz property test \cite[Thm.~3.4]{Mityagin-2025-31} to the case of arbitrary multiplicities of $\mu_k$, $k \in \N$.
\begin{theorem}\label{thm:RB}
Let $A$ be a self-adjoint operator in a separable Hilbert space and let the spectrum of $A$ satisfy \eqref{A.sp.intro}. 
Suppose that a sesquilinear form $v$ satisfies \eqref{asm:v.loc.sub} and \eqref{asm:omega} and let $T$ be the m-sectorial operator defined via the form sum of $A$ and $v$ in Lemma~\ref{lem:t.def}. 
Let $N_0 \in \N$ and the projections $S_0$ and $P_k$, $k>N_0$, be as in Proposition~\ref{prop:loc}.
Then the system $\{S_0\} \cup \{P_k\}_{k>N_0}$ is complete and there exists a boundedly invertible $W \in \cB(\cH)$ such that  
\begin{equation}\label{S0Pk.Riesz}
S_0 = W^{-1} \left(\sum_{k=1}^{N_0} P_k^0 \right) W, \qquad P_k = W^{-1} P_k^0 W, \quad k > N_0.
\end{equation}
\end{theorem}

\subsection{Proof of Theorem~\ref{thm:RB}}
The proof comprises two steps. First: a bound on the sum of the difference of spectral projections $P_n$ and $P_n^0$ is shown in Lemma~\ref{lem:Pn.bdd}. Second: the completeness of the system $\{S_0\} \cup \{P_n\}_{n>N_0}$ is established in Lemma~\ref{lem:complete}.

The first step is an adjustment of the proof of \cite[Thm.~3.4]{Mityagin-2025-31}, where a version of Lemma~\ref{lem:Pn.bdd} for projections with $\Rank(P_n)=1$, $n>N_0$, is established.

\begin{lemma}\label{lem:Pn.bdd}
	Let the assumptions of Theorem~\ref{thm:RB} be satisfied and let $N_0 \in \N$ be as in Proposition~\ref{prop:loc}. Then there exists $N_* \in \N$ with $N_* > N_0$ such that 
	\begin{equation}\label{RB.cond.N*}
		\sum_{n \geq N_*} |\langle (P_n - P_n^0) f, f \rangle|
		\leq 2 \|f\|^2.
	\end{equation}
\end{lemma}

\begin{proof}[Sketch of the proof]

We employ the representation of the resolvent difference (see \eqref{Tz.res.dec})
\begin{equation}\label{res.exp.s}
	(z-T)^{-1} - (z-A)^{-1} = K(z) \left(\sum_{s =0}^\infty B(z)^{s+1} \right) K(z) 
\end{equation}
if $z \in \rho(A)$ and $\|B(z)\|<1$. We note that $\|B(z)\| \leq 1/2$ for all $z \notin \Pi$ with $\Re z \geq \mu_{N} -  r_{N}$ for a sufficiently large $N$, see  \eqref{B.est.2n}. 

The crucial ingredient is the following bound. Namely, there exists $N_* \in \N$ with $N_* > N_0$ such that for each non-negative integer $s$
\begin{equation}\label{N*.s.sum}
\sum_{n \geq N_*} \left| \frac{1}{2 \pi \ii} \int_{\partial \Pi_n} \langle KB^{s+1} K f,f \rangle \; \dd z \right|
\leq 2^{-s}\|f\|^2, \quad f \in \cH.
\end{equation}

The proof \eqref{N*.s.sum} follows closely the estimates in \cite[Sec.~4.3]{Mityagin-2025-31}. The only needed adjustment is the representation of the quadratic form of $KB^{s+1}K$ adapted to the expansion with respect to the complete orthogonal projections $P_k^0$, $k \in \N$. To this end, we show below that the operator $B$ defined by \eqref{B.form.def} can be expressed also as
\begin{equation}\label{B.PVP}
B = \sum_{j,k \in \N} \frac{1}{(z-\mu_j)^\frac12(z-\mu_k)^\frac12} P_k^0 \wh V P_j^0.
\end{equation}
Here $\wh V$ is the bounded operator associated with the form $v$, more precisely,
\begin{equation}
\wh V : \cD_a:=(\Dom(A^\frac12), \|A^\frac12 \cdot\|) \to \cD_a^*:= (\Dom(A^\frac12), \|A^\frac12 \cdot\|)^*, \ \wh V f = v(f, \cdot); 
\end{equation}
we clearly have $\cD_a \hookrightarrow \cH \equiv \cH^* \hookrightarrow \cD_a^*$ where the embeddings are continuous and have dense ranges.
The projections $P_j^0$ are viewed as maps from $\cH$ to $\cD_a$. On the other hand, $P_k^0$ are viewed as maps from $\cD_a^*$ to $\cH$, i.e.~, for any $f^* \in \cD_a^*$ 
\begin{equation}
(P_k^0 f^*) (g) = f^*(P_k^0 g ), \quad g \in \cD_a.
\end{equation}
Indeed, we have $P_k^0 f^* \in \cH^* \equiv \cH$ as
\begin{equation}
\begin{aligned}
| (P_k^0 f^*) (g)| & \leq \|f^*\|_{\cD_a^*} \|P_k^0g\|_{\cD_a} = \|f^*\|_{\cD_a^*} \|A^\frac12 P_k^0 g \| 
= \mu_k^\frac12 \|f^*\|_{\cD_a^*} \|P_k^0 g\| 
\\
& \leq \mu_k^\frac12 \|f^*\|_{\cD_a^*} \|g\|.
\end{aligned}	
\end{equation}
Since for all $f,g \in \cH$,
\begin{equation}
| \langle P_k^0 \wh V P_j^0 f, g \rangle | = | v(P_j^0 f, P_k^0 g) | \leq \omega_j \omega_k \|P_j^0 f\| \|P_k^0 g\|, \quad j,k \in \N,
\end{equation}
the r.h.s.~of \eqref{B.PVP} defines a bounded operator on $\cH$; the estimates are analogous to \eqref{B.norm.detail}. Finally, it is easy to check that $B$ in \eqref{B.PVP} coincides with $B$ defined by \eqref{B.form.def}. 

A repeated application of \eqref{B.PVP} yields that for all $f \in \cH$ and $s \in \N_0$, 
\begin{equation}\label{KVK.s.exp}
	\begin{aligned}
		\langle KB^{s+1} K f,f \rangle 
		= \sum_{j,j_1, \dots, j_s, k \in \N} \frac{\langle P_k^0 \wh V P_{j_1}^0 P_{j_1}^0 \wh V P_{j_2}^0 \dots P_{j_{s-1}}^0 \wh V P_{j_s}^0  P_{j_{s}}^0 \wh V P_{j}^0  f,  P_k^0 f \rangle}{(z-\mu_j)(z-\mu_{j_1})(z-\mu_{j_2})\dots (z-\mu_{j_s})(z-\mu_k)};  
	\end{aligned}
\end{equation}
cf.~\cite[Lemma~4.3]{Mityagin-2025-31} in the case $\Rank(P_n^0)=1$ for all $n>N_0$. 

In the next step, \eqref{N*.s.sum} is justified by estimating the contour integrals $\int_{\partial \Pi_n} \langle KB^{s+1} K f,f \rangle \, \dd z$; the details are fully analogous to \cite[Lemma~4.4 -- 4.8]{Mityagin-2025-31}.

Returning to \eqref{res.exp.s} and employing \eqref{N*.s.sum}, we obtain
\begin{equation}
\begin{aligned}
\sum_{n \geq N_*} |\langle (P_n - P_n^0) f, f \rangle| & = 
\sum_{n \geq N_*} \left| \frac{1}{2 \pi \ii} \int_{\partial \Pi_n} \Big\langle 
K \Big(\sum_{s =0}^\infty B^{s+1} \Big) K f,f 
\Big \rangle \; \dd z \right|
\\
& \leq 
\sum_{s=0}^\infty 2^{-s} \|f\|^2 = 2 \|f\|^2, \qquad f \in \cH,
\end{aligned}
\end{equation}
which proves \eqref{RB.cond.N*}.
\end{proof}

In the second step, the completeness of the system of projections \eqref{Pn.SN.def} is proved using the strategy in \cite[Proof~of Thm.~1.1]{Motovilov-2017-8}.

\begin{lemma}\label{lem:complete}
Let the assumptions of Theorem~\ref{thm:RB} be satisfied. Then the system of projections \eqref{Pn.SN.def} is complete.
\end{lemma}
\begin{proof}
Let $\cR_n$ be the boundary of the square $[-\gamma_n, \gamma_n] \times [-\ii \gamma_n, \ii \gamma_n] $ in $\C$ where 
\begin{equation}\label{gamman.def}
	\gamma_n : = \mu_n + r_n,
\end{equation}
and let for $n>N_0$
\begin{equation}\label{In0.In.def}
	\begin{aligned}
		S_n^0 &:= \frac1{2 \pi \ii} \int_{\cR_n} (z-A)^{-1} \, \dd z = \sum_{k=1}^n P_k^0,
		\\
		S_n &:= \frac1{2 \pi \ii} \int_{\cR_n} (z-T)^{-1} \, \dd z = S_0 + \sum_{k=N_0+1}^n P_k.
	\end{aligned}
\end{equation}
Since $A=A^*$, we have for each $f \in \cH$ that
$
	\lim_{n \to \infty} S_n^0 f = f.
$
We establish below that also $\lim_{n \to \infty} S_n f = f$ for each $f \in \cH$, thus \eqref{Pn.SN.def} is complete.

In the first step, we show that there exists $C >0$ such that for all $n  > N_0$
	\begin{equation}\label{In.bdd}
		\|S_n^0 - S_n\| \leq C.
	\end{equation}

Employing \eqref{In0.In.def}, it follows from \eqref{RB.cond.N*}  that for all $n > N_*$
	\begin{equation}
		\begin{aligned}
			|\langle (S_n^0 - S_n)f, f\rangle| 
			& \leq \left\|\sum_{k=1}^{N_0} P_k^0 -S_0 \right \|\|f\|^2 
			+
			\sum_{k = N_0 +1}^{N_*-1} \|P_k - P_k^0\| \|f\|^2
			\\
			& \qquad +
			\sum_{k \geq N_*} |\langle (P_k - P_k^0) f, f \rangle|
			\\
			& \leq 
			\left(
			1 + \|S_0\|+ \sum_{k = N_0 +1}^{N_*-1} \|P_k - P_k^0\| + 2
			\right) 
			\|f\|^2.
		\end{aligned}
	\end{equation}
	Hence \eqref{In.bdd} can be justified by polarization, see e.g.~\cite[Lemma~IV.2.1]{EE}.

	In the second step, we show that $\lim_{n \to \infty} (S_n^0 - S_n) f = 0$ for each $f \in \cH$.
	To this end, due to the uniform boundedness \eqref{In.bdd}, it suffices to verify that for each $m_0 \in \N$ 
	\begin{equation}\label{In.psim0}
		\lim_{n \to \infty} (S_n^0 - S_n) P_{m_0}^0 f = 0
	\end{equation}
	since $\{P_m^0\}_{m \in \N}$ is complete.

	From the representation of the resolvent difference (see \eqref{Tz.res.dec})
	\begin{equation}\label{res.dif}
		(z-T)^{-1} - (z-A)^{-1} =  K(z) B(z) (I-B(z))^{-1}K(z)
	\end{equation}  
	and 
	\begin{equation}
		\|K(z) P_{m_0}^0 f \| = \frac{\|P_{m_0}^0 f\|}{|z-\mu_{m_0}|^\frac 12},
	\end{equation}
	we obtain that
	\begin{equation}\label{int.In}
	\|(S_n^0 - S_n ) P_{m_0}^0 f \| 
		\leq \frac{1}{2\pi} \int_{\cR_n } \|K(z)\| \|B(z)\| \|(I-B(z))^{-1}\| \frac{\|P_{m_0}^0 f\|}{|z-\mu_{m_0}|^\frac 12} \, \dd z.
	\end{equation}
	Thus there exists $M_0 \equiv M_0(m_0)\geq N_*$ such that for all $n > M_0$
	\begin{equation}\label{int.In.2}
		\|(S_n^0 - S_n )P_{m_0}^0 f \| 
		\leq \frac{\|f\|}{\pi} \int_{\cR_n } \|K(z)\| \|B(z)\| \|(I-B(z))^{-1}\| \frac{1}{|z|^\frac 12} \, \dd z.
	\end{equation}
	
	Let $\eps \in (0,\frac 12)$. From \eqref{B.est.1} and \eqref{B.est.2n}, we obtain that there exists $N_1 \geq M_0$ such that 
	\begin{equation}\label{B.eps.1}
		\sup_{\Re z \leq -\gamma_{N_1}} \|B(z)\| < \eps \qquad \text{and} \qquad \sup_{\Re z \geq \gamma_{N_1}, z \notin \Pi} \|B(z)\| < \eps;
	\end{equation}
	moreover, for the already fixed $N_1$, from \eqref{B.est.3}, there exists $N_2 > N_1$ such that 
	\begin{equation}\label{B.eps.2}
		\sup_{|\Re z| \leq \gamma_{N_1}, |\Im z| \geq \gamma_{N_2}} 
		\|B(z)\| < \eps.
	\end{equation}
	
	Finally we estimate the integral in \eqref{int.In.2} for $n > N_2$. For the vertical sides of $\cR_n$, on the left we have
	\begin{equation}
		\begin{aligned}
			&\int_{-\gamma_n}^{\gamma_n} \|K(-\gamma_n + \ii \eta)\| \|B(-\gamma_n + \ii \eta)\| \|(I-B(-\gamma_n + \ii \eta))^{-1}\| \frac{1}{|-\gamma_n + \ii \eta|^\frac 12} \, \dd \eta	
			\\
			& \quad \leq
			\frac{\eps}{1-\eps} \int_{-\gamma_n}^{\gamma_n}  \frac{1}{|\gamma_n + \ii \eta|} \, \dd \eta
			= 
			\frac{2 \eps}{1-\eps} \int_{0}^{1}  \frac{1}{(1 + t^2)^\frac12} \, \dd t \leq 4 \eps,
		\end{aligned}
	\end{equation}
	and on the right
	\begin{equation}
		\begin{aligned}
			&\int_{-\gamma_n}^{\gamma_n} \|K(\gamma_n + \ii \eta)\| \|B(\gamma_n + \ii \eta)\| \|(I-B(\gamma_n + \ii \eta))^{-1}\| \frac{1}{|\gamma_n + \ii \eta|^\frac 12} \, \dd \eta	
			\\
			& \quad \leq
			\frac{\eps}{1-\eps} \int_{-\gamma_n}^{\gamma_n}  \frac{1}{|\eta|^{\frac 12}|\gamma_n + \ii \eta|^\frac12} \, \dd \eta
			= 
			\frac{2\eps}{1-\eps} \int_{0}^{1}  \frac{1}{t^\frac12 (1 + t^2)^\frac14} \, \dd t \leq 8 \eps.
		\end{aligned}
	\end{equation}
	For the horizontal sides of $\cR_n$, we obtain 
	\begin{equation}\label{int.hor}
		\begin{aligned}
			&\int_{-\gamma_n}^{\gamma_n} \|K(\xi \pm \ii \gamma_n)\| \|B(\xi \pm \ii \gamma_n)\| \|(I-B(\xi \pm \ii \gamma_n))^{-1}\| \frac{1}{|\xi \pm \ii \gamma_n|^\frac 12} \, \dd \xi	
			\\
			& \quad =
			\int_{-\gamma_{N_1}}^{\gamma_{N_1}} \|K(\xi \pm \ii \gamma_n)\| \|B(\xi \pm \ii \gamma_n)\| \|(I-B(\xi \pm \ii \gamma_n))^{-1}\| \frac{1}{|\xi \pm \ii \gamma_n|^\frac 12} \, \dd \xi	
			\\
			&  \qquad +
			\int_{\gamma_{N_1} \leq |\xi| \leq \gamma_{n}} \|K(\xi \pm \ii \gamma_n)\| \|B(\xi \pm \ii \gamma_n)\| \|(I-B(\xi \pm \ii \gamma_n))^{-1}\| \frac{1}{|\xi \pm \ii \gamma_n|^\frac 12} \, \dd \xi.
		\end{aligned}
	\end{equation}
	Next we employ \eqref{B.eps.2} and \eqref{B.eps.1}; notice that if $\xi \in [-\gamma_n,\gamma_n]$, then $\xi + \ii \gamma_n \notin \Pi$ as $r_n < \omega_n$, $n>1$, see \eqref{rk.def}.  Thus the r.h.s.~of \eqref{int.hor} does not exceed
	\begin{equation}
	\begin{aligned}
		& \frac{\eps}{1-\eps} \int_{-\gamma_{N_1}}^{\gamma_{N_1}} \frac1{\gamma_n^\frac12} \frac{1}{|\xi + \ii \gamma_n|^\frac 12} \, \dd \xi	
		+
		\frac{\eps}{1-\eps} \int_{\gamma_{N_1} \leq |\xi| \leq \gamma_{n}} \frac1{\gamma_n^\frac12} \frac{1}{|\xi + \ii \gamma_n|^\frac 12} \, \dd \xi
		\\
		& \qquad \qquad = 
		\frac{2\eps}{1-\eps} \int_{0}^{\gamma_n} \frac1{\gamma_n^\frac12} \frac{1}{|\xi + \ii \gamma_n|^\frac 12} \, \dd \xi
		=
		\frac{2\eps}{1-\eps} \int_{0}^{1} \frac{1}{(t^2 + 1)^\frac 14} \, \dd t
		\leq 4 \eps.
	\end{aligned}
	\end{equation}
	Since $\eps \in (0,\frac 12)$ was arbitrary, \eqref{In.psim0} indeed holds.
\end{proof}

\begin{proof}[Proof of Theorem~\ref{thm:RB}]
Since the system of projections $\{S_0\}\cup\{P_n\}_{n > N_0}$ is complete, the estimate \eqref{RB.cond.N*} justifies its the Riesz property, see \cite[Chap.~6]{Gohberg-1969}, \cite[\S 6]{Shkalikov-2016-71}, \cite[Thm.~A]{Motovilov-2017-8} or \cite[Thm.~2.3]{Mityagin-2025-31}.
\end{proof}


\section{Perturbations of differential operators with multiple eigenvalues}
\label{sec:diff.op}

In the application to differential operators below, we consider perturbations by potentials $V \in L^r$, $r \geq 1$. Recall that in this case, the H\"older inequality in \eqref{V.Hold} with \eqref{V.hold.rp} yields that the sequence $\{\omega_j\}_{j \in \N}$ in \eqref{asm:v.loc.sub} can be chosen as $\|P_j^0\|_{L^2 \to L^p}$

\subsection{Multi-D harmonic oscillator}

Let $A_{\rm HO}$ be the $d$-dimensional harmonic oscillator
\begin{equation}\label{A.HO.def}
A_{\rm HO} = - \Delta + |x|^2, \quad \Dom(A_{\rm HO}) = H^2(\Rd) \cap \Dom(|x|^2)
\end{equation}
in $L^2(\Rd)$. The operator $A_{\rm HO}$ is self-adjoint, has compact resolvent, its eigenvalues read 
\begin{equation}
\mu_k = 2 k -2  + d, \quad r_k =1, \quad k \in \N,
\end{equation}
and have the multiplicities  
\begin{equation}\label{rank.HO}
\Rank(P_k^0) = \binom{k-2 + d}{k-1}, \quad k \in \N;
\end{equation}
the eigenfunctions can be expressed in terms of Hermite functions, see e.g.~\cite[Chap.~18.1]{Weidmann-2003}.

\begin{theorem}\label{thm:HO}
Let $A_{\rm HO} = - \Delta + |x|^2 $ be the harmonic oscillator \eqref{A.HO.def} in $L^2(\Rd)$ with $d \geq 2$ and let $V \in L^r(\Rd)$ with 
\begin{equation}\label{r.cond.HO}
	\frac d2 < r < \infty. 
\end{equation}
Then the form of $V$ satisfies the conditions \eqref{asm:v.loc.sub} and \eqref{asm:omega} with 
\begin{equation}
\omega_j = 
\begin{cases}
\BigO\left(j^{- \frac 1{4r}} \right), &  0< \frac 1r  <\frac{2}{d+3},
\\[2mm]
\BigO \left( j^{- \frac 12 \frac{1}{d+3} } (\log j)^\frac{d+1}{2(d+3)} \right) , & \frac 1r = \frac{2}{d+3},
\\[2mm]
\BigO \left(j^{-\frac d{12}( \frac 2d  - \frac 1r)} \right),& \frac{2}{d+3} < \frac 1r < \frac 2d.
\end{cases}
\end{equation}
Hence the eigenvalues of $T= - \Delta + |x|^2+V$ are localized as in Proposition~\ref{prop:loc} or Remark~\ref{rem:EV.loc} and the system of spectral projections of $T$ has a Riesz property (see Theorem~\ref{thm:RB}).
\end{theorem}
\begin{proof}
For $d \geq 2$, it is known that there exists $C \equiv C(p,d) >0$, independent of $k$, such that for all $k \geq 2$
\begin{equation}\label{herm.Lp}
\|P_k^0\|_{L^2 \to L^p} \leq C k^{\frac 12 \rho_{\rm HO}(p)},		
\end{equation}
where
\begin{equation}\label{herm.Lp.2}
\rho_{\mathrm{HO}}(p) = \begin{cases}
-[ 1 - d (\frac 12 - \frac 1p)], &  0 \leq \frac 1p \leq \frac12 - \frac1d,
\\[1mm]
- \frac13[1 - d(\frac 12 - \frac 1p)], & \frac12 - \frac1d  \leq \frac 1 p < \frac12 - \frac1{d+3}, 
\\[1mm]
- (\frac 12 - \frac 1p), & \frac12 - \frac1{d+3} < \frac 1p \leq \frac 12,
\end{cases}
\end{equation}
see \cite[Cor.~3.2]{Koch-2005-128}, and for $p=p^*$, $\frac 1{p^*} = \frac12 - \frac1{d+3}$, the result  \cite[Thm.~2]{Koch-2007-8} yields
\begin{equation}\label{herm.Lp.3}
\|P_k^0\|_{L^2 \to L^{p^*} } \leq C k^{- \frac 12 (\frac 12 - \frac 1{p^*}) } (\log k)^{\frac 1{p^*}},		
\end{equation}
see also \cite{Wang-2025-378}. In the case $p=p^*$, for every $\eps>0$, there exists $C_\eps >0$ such that for all $k \geq 2$
\begin{equation}\label{herm.Lp.3.eps}
\|P_k^0\|_{L^2 \to L^{p^*}} \leq C_\eps k^{- \frac 12 (\frac 12 - \frac 1{p^*}) +\eps}.		
\end{equation}
(Let us notice that \eqref{herm.Lp.3.eps} can be obtained from the second and third line in \eqref{herm.Lp.2} by Cauchy-Schwarz inequality with $p = p^* + \delta$ and $p = p^* - \delta$ respectively and $\delta \equiv \delta(\epsilon) > 0$ small enough.)

Let $V \in L^r(\Rd)$ and recall \eqref{V.Hold} -- \eqref{Pj.ue} and Example~\ref{ex:mu.om.power}. Since $\mu_{k+1}-\mu_k = 2$, i.e., $\gamma = 1$ in \eqref{omega.reg.1}, the sufficient condition \eqref{omega.reg.2} holds if
\begin{equation}\label{rho.HO.cond}
\rho_{\mathrm{HO}}(p)<0.
\end{equation}
For $\frac 1{2r} = \frac12 - \frac 1p$, we obtain from \eqref{herm.Lp.2} that 
\begin{equation}\label{herm.Lp.2.r}
\rho_{\mathrm{HO}}(p) = 
\begin{cases}
- \frac1{2r}, & 0 \leq \frac 1r < \frac{2}{d+3},
\\[1mm]
- \frac d6( \frac 2d  - \frac 1r), & \frac{2}{d+3} < \frac 1r \leq \frac 2d,
\\[1mm]		
-\frac d2(  \frac2d- \frac 1r), &   \frac 2d \leq \frac 1r \leq 1.
\end{cases}
\end{equation}
Thus \eqref{omega.reg.1} with \eqref{omega.reg.2} hold if $1/r <2/d$; notice that the special case with $1/r = 2/(d+3)$ is included as well, cf.~\eqref{herm.Lp.3.eps}.
\end{proof}

Since the gaps in $\sigma(A_{\rm HO})$ are equal to $2$, by \cite[Thm.~1]{Motovilov-2017-8}, the statements of
Proposition~\ref{prop:loc} and Theorem~\ref{thm:RB} hold for $- \Delta + |x|^2+V$ also if $V \in L^\infty(\Rd)$ with $\|V\|_{L^\infty} <1$.	

Finally, we remark that for $V \in L^r(\Rd)$ with compact support one can employ the local estimates on the spectral projections of $A_{\rm HO}$ in \eqref{A.HO.def} obtained recently in \cite[Thm.~2]{Wang-2025-378}. Nonetheless, this does not lead to a weaker restriction on $r$ than \eqref{r.cond.HO}.

\subsection{Landau Hamiltonian (twisted Laplacian) in even dimensions}
\label{ssec:Landau}

Let $d \in 2\N$. We write $(x,y) \in \Rd$ where $x,y \in \R^{d/2}$ and consider the vector potential
\begin{equation}
\mathbf A(x,y) = \frac 12 (- y_1, x_1, \dots, - y_{d/2},x_{d/2}).
\end{equation}
The Landau Hamiltonian, see \cite{Landau-1930-64}, in $L^2(\Rd)$ reads
\begin{equation}\label{A.Landau.def}
\begin{aligned}
A_{\rm Lan} & = (-\ii \nabla + \mathbf A)^2 = \sum_{j=1}^\frac d2 
\left(
-\ii \partial_{x_j} + \frac{y_j}{2}
\right)^2 + 
\left(
-\ii \partial_{y_j} + \frac{x_j}{2}
\right)^2,
\\
\Dom(A_{\rm Lan}) & = \{ f \in H_A^1(\Rd) \, : \, (-\ii \nabla + \mathbf A)^2f \in L^2(\Rd)\},
\end{aligned}
\end{equation}
where $H_A^1(\Rd) = \{ f \in L^2(\Rd) \, : \, |(-\ii \nabla + \mathbf A)f| \in L^2(\Rd)  \}$ is the magnetic Sobolev space; for some recent related works see \cite{Koch-2007-180,Cuenin-2017-42,Cuenin-2017-88} and for more details on magnetic Laplacian see \cite{Raymond-2017}. The operator $A_{\rm Lan}$ is self-adjoint, it has purely point spectrum, its the eigenvalues read
\begin{equation}
	\mu_k = 2 k -2 + \frac d2, \quad r_k=1, \quad k \in \N, 
\end{equation}
and $\Rank(P_k^0) = \infty$, $k \in \N$. The eigenspaces of $A_{\rm Lan}$ can be described using Hermite functions, see e.g.~\cite{Koch-2007-180} for details.

\begin{theorem}\label{thm:Landau}
Let $A_{\rm Lan} = (-\ii \nabla + \mathbf A)^2$ be the Landau Hamiltonian \eqref{A.Landau.def} in $L^2(\Rd)$ with $d \in 2 \N$ and let $V \in L^r(\Rd)$ with 
\begin{equation}\label{r.cond.Landau}
	\frac d2 < r < \infty. 
\end{equation}
Then the form of $V$ satisfies the conditions \eqref{asm:v.loc.sub} and \eqref{asm:omega} with 
\begin{equation}
	\omega_j = 
	\begin{cases}
		\BigO\left(j^{- \frac 1{4r}} \right), &  0< \frac 1r  \leq \frac{2}{d+1},
		\\[2mm]
		\BigO \left(j^{ -\frac d4(\frac 2d - \frac 1r)}\right),& \frac{2}{d+1} < \frac 1r < \frac 2d.
	\end{cases}
\end{equation}
Hence the spectrum of $T = (-\ii \nabla + \mathbf A)^2 +V$ is localized as in Proposition~\ref{prop:loc} or Remark~\ref{rem:EV.loc} and the system of spectral projections of $T$ has a Riesz property (see Theorem~\ref{thm:RB}).
\end{theorem}
\begin{proof}
It is known that there exists $C \equiv C(p,d) >0$, independent of $k$, such that for all $k \geq 1$ 
\begin{equation}\label{Landau.Lp}
\|P_k^0\|_{L^2 \to L^p} \leq C k^{\frac 12 \rho_{\mathrm{Lan}}(p)} 
\end{equation}
where
\begin{equation}\label{Landau.Lp.1}
	\rho_{\mathrm{Lan}}(p) = \begin{cases}
	-1 + d (\frac 12 - \frac 1p), &  0 \leq \frac 1p \leq \frac 12 - \frac 1{d+1},
	\\[1mm]
	- (\frac 12 - \frac 1p), & \frac 12 - \frac 1{d+1} \leq \frac 1p \leq \frac 12 ,
	\end{cases}
\end{equation}
see \cite[Thm.~1.2]{Stempak-1998-76} and \cite[Thm.~1]{Koch-2007-180}.

Let $V \in L^r(\Rd)$  and recall \eqref{V.Hold} -- \eqref{Pj.ue} and Example~\ref{ex:mu.om.power}. Since $\mu_{k+1}-\mu_k = 2$, i.e.,~$\gamma = 1$ in \eqref{omega.reg.1}, the sufficient condition \eqref{omega.reg.2} holds if
\begin{equation}\label{rho.Lan.cond}
	\rho_{\mathrm{Lan}}(p)<0.
\end{equation}
For $\frac 1{2r} = \frac 12- \frac1p$, we obtain from \eqref{Landau.Lp.1} that 
\begin{equation}\label{Landau.Lp.r}
	\rho_{\mathrm{Lan}}(p) = 
	\begin{cases}
		- \frac1{2r}, & 0 \leq \frac 1r \leq \frac 2{d+1}, 
		\\[1mm]
		-\frac d2(\frac 2d - \frac 1r), &  \frac 2{d+1} \leq \frac 1r \leq 1.	
	\end{cases}
\end{equation}
Thus \eqref{omega.reg.1} with \eqref{omega.reg.2} hold if $1/r< 2/d$.
\end{proof}

Since the gaps in $\sigma(A_{\rm Lan})$ are equal to $2$, by \cite[Thm.~1]{Motovilov-2017-8}, the statements of Proposition~\ref{prop:loc} and Theorem~\ref{thm:RB} hold for $(-\ii \nabla + \mathbf A)^2+V$ also if $V \in L^\infty(\Rd)$ with $\|V\|_{L^\infty} <1$.	
 
\subsection{Laplace-Beltrami operator on $\mathbb S^d$}
\label{ssec:LB}
Let $d >1$ and $A_{\mathbb S^d}$ be the Laplace-Beltrami operator on the sphere $\mathbb S^d$   
\begin{equation}\label{LB.def}
A_{\mathbb S^d} = - \Delta_{\mathbb S^d}, \quad \Dom(A_{\mathbb S^d} ) = H^2({\mathbb S^d});
\end{equation}
see e.g.~\cite[Chap.~18.2]{Weidmann-2003}. The operator $A_{\mathbb S^d}$ is self-adjoint, has compact resolvent, its eigenvalues read
\begin{equation}\label{LB.mu.k}
\mu_k = (k-1)(k-2+d), \quad r_k = k - 2 + \frac d2, \quad k \in \N,
\end{equation}
and have multiplicities  
\begin{equation}
\Rank(P_k^0) = \frac{(2k-3+d)(k-3+d)!}{(k-1)!(d-1)!}, \quad k \in \N, \ d>1;
\end{equation}
see e.g.~\cite[Satz.~18.11]{Weidmann-2003}. The eigenfunctions can be expressed in terms of spherical harmonics, see e.g.~\cite[Chap.~18.2]{Weidmann-2003}.

\begin{theorem}\label{thm:LB}
Let $A_{\mathbb S^d} = - \Delta_{\mathbb S^d}$ be the Laplace-Beltrami operator on $\mathbb S^d$ in \eqref{LB.def} in $L^2(\mathbb S^d)$ with $d >1$ and let $V \in L^r(\mathbb S^d)$ with 
\begin{equation}\label{r.cond.LB}
\frac d2 < r \leq \infty. 
\end{equation}
Then the form of $V$ satisfies the conditions \eqref{asm:v.loc.sub} and \eqref{asm:omega} with
\begin{equation}
	\omega_j = 
	\begin{cases}
		\BigO\left(j^{\frac12 -\frac 1{4r} - \frac d4 ( \frac 2d - \frac 1r)} \right), &  0 \leq \frac 1r \leq \frac2{d+1},
		\\[2mm]
		\BigO \left(j^{\frac12 - \frac d2(\frac 2d - \frac 1r)} \right),&  \frac2{d+1} < \frac 1r < \frac{2}{d} .
	\end{cases}
\end{equation}
Hence the eigenvalues of $ T = - \Delta_{\mathbb S^d} + V $ are localized as in Proposition~\ref{prop:loc} or Remark~\ref{rem:EV.loc} and the system of spectral projections of $T$ has a Riesz property (see Theorem~\ref{thm:RB}).

\end{theorem}
\begin{proof}
It is known that there exists $C \equiv C(p,d) >0$, independent of $k$, such that for all $k \geq 1$
\begin{equation}\label{Sphere.Lp}
\|P_k^0\|_{L^2(\mathbb S^d) \to L^p(\mathbb S^d)} 
\leq C k^{\rho_{\mathrm{Sph}}(p)},		
\end{equation}
where
\begin{equation}\label{Sphere.Lp.2}
\rho_{\mathrm{Sph}}(p) =\begin{cases}
- \frac 12 + d(\frac 12 - \frac 1p), &  0 \leq \frac 1p \leq \frac12 - \frac1{d+1},
\\[1mm]
\frac{d-1}{2} (\frac 12 - \frac 1p), & \frac12 - \frac1{d+1}  \leq \frac 1p \leq \frac 12,		
\end{cases}
\end{equation}
see \cite{Sogge-1988-77} (or also e.g.~\cite{Wang-2025-378}).

Let $V \in L^r(\Rd)$  and recall \eqref{V.Hold} -- \eqref{Pj.ue} and Example~\ref{ex:mu.om.power}. The gaps $\mu_k$ satisfy
\begin{equation}\label{mu.k.gaps.Sd}
	\mu_{k+1} - \mu_k \geq 2 k -2 + d, \quad k \in \N, 
\end{equation}
so \eqref{omega.reg.1} holds with $\gamma=2$. Thus the condition \eqref{omega.reg.2} is satisfied if $- 2 \rho_{\mathrm{Sph}}(p) + 2 >1 $
or equivalently
\begin{equation}\label{Sphere.Lp.3}
\rho_{\mathrm{Sph}}(p) - \frac 12 <0.
\end{equation}
For $\frac 1{2r} = \frac 12- \frac1p$, we obtain from \eqref{Sphere.Lp.2} that 
\begin{equation}\label{Sphere.Lp.r}
\rho_{\mathrm{Sph}}(p) - \frac12 = 
\begin{cases}
-\frac 1{4r} - \frac d4 ( \frac 2d - \frac 1r), &   0 \leq \frac 1r \leq \frac2{d+1},
\\[1mm]
- \frac d2(\frac 2d - \frac 1r), &  \frac2{d+1} \leq \frac 1r \leq 1 .
\end{cases}
\end{equation}
Hence \eqref{omega.reg.1} with \eqref{omega.reg.2} hold if \eqref{r.cond.LB} is satisfied. 
\end{proof}

In the next result, we perturb $-\Delta_{\mathbb S^d}$ by a $\delta_\Sigma$-potential supported on a $(d-1)$-dimensional submanifold $\Sigma$ with a coupling $W \in L^r(\Sigma)$. Technically, this amounts to the perturbation by the form
\begin{equation}\label{v.delta.def}
	v(f,g) = \int_\Sigma W(x) f(x) \ov{g(x)} \, \dd \sigma(x).
\end{equation}

\begin{theorem}\label{thm:LB.delta}
Let $A_{\mathbb S^d}$ be the Laplace-Beltrami operator on $\mathbb S^d$ in \eqref{LB.def} in $L^2(\mathbb S^d)$ with $d >1$. Let $\Sigma$ be a $(d-1)$-dimensional submanifold of $\mathbb S^d$ and let $W \in L^r(\Sigma)$ with 
\begin{equation}\label{r.cond.LB.delta}
d-1 < r \leq \infty. 
\end{equation}
Then the form $v$ in \eqref{v.delta.def} satisfies the conditions \eqref{asm:v.loc.sub} and \eqref{asm:omega} with
	\begin{equation}
		\omega_j = 
		\begin{cases}
			\BigO\left(j^{	\frac 14  + \frac{d-2}{4} \frac 1r} \right), &  0 \leq \frac 1r < \frac 1d,
			\\[2mm]
			\BigO\left( j^{\frac{d-1}{2d}} (\log j)^\frac 12 \right), &  \frac 1r = \frac 1d,
			\\[2mm]
			\BigO \left(j^{ \frac{d-1}{2} \frac 1r } \right),&  \frac1 d < \frac 1r < \frac{1}{d-1}.
		\end{cases}
	\end{equation}
 Hence the eigenvalues of $ T = - \Delta_{\mathbb S^d} + W \delta_\Sigma $ are localized as in Proposition~\ref{prop:loc} or Remark~\ref{rem:EV.loc} and the system of spectral projections of $T$ has a Riesz property (see Theorem~\ref{thm:RB}).
\end{theorem}
\begin{proof}
It is known that there exists $C \equiv C(p,d) >0$, independent of $k$, such that for all $k \geq 2$
\begin{equation}\label{Sphere.Lp.new}
\|P_k^0\|_{L^2(\mathbb S^d) \to L^p(\Sigma)} 
\leq C k^{\rho_{\Sigma}(p)},		
\end{equation}
where
\begin{equation}\label{Sphere.Sigma}
\rho_{\Sigma}(p)= 
\begin{cases}
\frac{d-1}{2} (1 - \frac 2p), &  0 \leq \frac 1p < \frac12 - \frac1{2d},
\\[1mm]
\frac{d-1}{4} - \frac{d-2}{2p} , & \frac12 - \frac1{2d} < \frac 1p \leq \frac 12,		
\end{cases}
\end{equation}
and for $p = p^*$, $\frac 1{p^*} = \frac 12 - \frac1{2d}$,
\begin{equation}\label{Sphere.Lp.new.2}
\|P_k^0\|_{L^2(\mathbb S^d) \to L^{p^*}(\Sigma)} 
\leq C k^{\frac{d-1}{2} (1 - \frac 2{p^*})} (\log k)^\frac 12,		
\end{equation}
see \cite[Thm.~3]{Burq-2007-138}.

For $v$ in \eqref{v.delta.def}, the H\"older inequality yields 
\begin{equation}
|v(f,g)| \leq \|W\|_{L^r(\Sigma)} \|f\|_{L^p(\Sigma)} \|g\|_{L^p(\Sigma)}
\end{equation}
where
\begin{equation}\label{rp.Sigma}
	\frac{1}{r} + \frac{2}{p} = 1.
\end{equation}
Hence the sequence $\{\omega_j\}_{j \in \N}$ in \eqref{asm:v.loc.sub} can be chosen as 
$ \|P_j^0\|_{L^2 \to L^p(\Sigma)}$. 

In view of Example~\ref{ex:mu.om.power} and \eqref{mu.k.gaps.Sd}, i.e., $\gamma = 2$ in \eqref{omega.reg.1}, the condition \eqref{omega.reg.2} is satisfied if $- 2 \rho_\Sigma(p) + 2 > 1$ or equivalently
\begin{equation}\label{ag.cond.Sd.Sigma}
\rho_\Sigma(p) < \frac 12.
\end{equation}
For $\frac 1{r} = 1- \frac2p$, we rewrite \eqref{Sphere.Sigma} as
\begin{equation}\label{Sphere.Sigma.r}
\rho_{\Sigma}(p)= 
\begin{cases}
\frac 14  + \frac{d-2}{4} \frac 1r , & 0 \leq \frac 1r < \frac 1d,		
\\[1mm]
\frac{d-1}{2} \frac 1r, &  \frac 1d < \frac 1r \leq 1.
\end{cases}
\end{equation}
Hence \eqref{omega.reg.1} and \eqref{omega.reg.2} hold if $r \in (d-1, \infty]$. The special case with $r = 1/d$ is included as well since \eqref{Sphere.Lp.new.2} or the first line of \eqref{Sphere.Sigma} and $|\Sigma|<\infty$ yield that for every $\eps>0$, there exists $C_\eps >0$ such that for all $k \geq 2$
\begin{equation}
\|P_k^0\|_{L^2(\mathbb S^d) \to L^{p^*}(\Sigma)} 
\leq C_\eps k^{\frac{d-1}{2} (1 - \frac 2{p^*})+\eps}.
\end{equation}
\end{proof}


\printbibliography


\end{document}